\documentclass{amsart}

\usepackage{graphicx}
\usepackage{amssymb,latexsym}
\usepackage{amsmath,amsfonts,amstext,bbm}
\usepackage[utf8]{inputenc}

\newtheorem{theorem}{Theorem}
\newtheorem{lemma}{Lemma}

\newtheorem{conjecture}{Conjecture}

\begin{document}
\title[Fractional integrals, lattice points, and diophantine approximation ]
      {Discrete fractional integrals, lattice points on short arcs, and diophantine approximation }
      
\author{Faruk Temur}
\address{Department of Mathematics\\
        Izmir Institute of Technology\\ Urla \\Izmir \\ 35430\\Turkey  }
\email{faruktemur@iyte.edu.tr}

\date{December  19, 2020}

\begin{abstract}
 Recently in  joint work   with E. Sert, we proved sharp boundedness results on discrete fractional integral operators along binary quadratic forms.  Present work vastly enhances the scope of those results  by  extending boundedness to
      bivariate  quadratic  polynomials.  We achieve this in part by  establishing  connections  to problems on   concentration of lattice points on short arcs of conics, whence  we study discrete fractional integrals and lattice point concentration from a unified perspective via tools of sieving and diophantine approximation, and prove theorems that are of interest to researchers in both subjects.
 \end{abstract}
\maketitle

\section{Introduction}

Let $f:\mathbb{Z}^l \rightarrow \mathbb{C}$ be a function and $P:\mathbb{Z}^{k+l} \rightarrow \mathbb{Z}^l$ be a polynomial with integer coefficients. Let $\mathbb{Z}^k_*=\mathbb{Z}^k-\{0\}$, and $0< \lambda\leq 1$.  We call
\begin{equation}\label{1.1}
\mathcal{I}_{\lambda,P}f(n)=\sum_{m\in \mathbb{Z}^k_*} \frac{f(P(m,n))}{|m|^{\lambda k}}
\end{equation}
    a discrete fractional integral operator, and  $P(m,n)$  the phase polynomial. When $P(m,n)=n-m$ the range of boundedness of these operators is the same as that of their continuous counterparts, and is given by Hardy-Littlewood-Sobolev inequality.  But when $P(m,n)$ contains higher order terms the two  cases differ significantly.  This phenomenon has generated immense interest in the last thirty years, see \cite{lbp} for an account.
    
    To investigate this phenomenon the most natural case to consider is the translation invariant case $P(m,n)=n-Q(m)$. Another case that has a similar flavor   is the  quasi-translation invariant case
    \begin{equation*}\label{}
    \mathcal{J}_{\lambda}f(n,n')=\sum_{m\in \mathbb{Z}^k_*} \frac{f(n-m,n'-Q(m,n))}{|m|^{\lambda k}},
    \end{equation*}
    where $f:\mathbb{Z}^{k+l} \rightarrow \mathbb{C}$, $Q:\mathbb{Z}^{k+k} \rightarrow \mathbb{Z}^l$.  These two cases  owing to  applicability of  Fourier analysis are relatively accessible, and have been studied intensely from this point of view for three decades, for a summary of the results see \cite{lbp,ts}.  The translation invariant case has also been studied from an alternative point of view by Oberlin \cite{ob} that uses arithmetic, and in particular representation of integers as sums of squares, rather than Fourier analysis.  But up  until the recent work of the author with E. Sert  \cite{ts},  no work existed on operators with  neither translation invariant nor  quasi-translation invariant phase polynomials.  In that work using arithmetic  extensively,  we proved first instances of such results. We start the discussion of results of this article by giving a brief summary of that work. To this end we introduce the relevant notation

    Let $P:\mathbb{Z}^2\rightarrow \mathbb{Z}$ be a  bivariate  quadratic  polynomial of integral coefficients, that is 
    \begin{equation*}
   P(m,n)=am^2+bmn+cn^2+dm+en+f,  \ \ \ \ \ \ \  \quad a,b,c,d,e,f,m,n \in \mathbb{Z}
        \end{equation*}
   with at least one of $a,b,c$ nonzero.  We call $\Delta(P):=b^2-4ac$ the discriminant of the polynomial, and we also define  the quantities $\alpha(P):= 2cd-be, \ \beta(P):=2ae-bd$. When the polynomial $P$ is clear from the context we will just write $\Delta, \alpha, \beta$.  If $d=e=f=0$,  the polynomial is called an integral binary quadratic form.   Henceforth we will exclusively concentrate on such polynomials and forms, and indeed  use the terms polynomial  and form to refer to them. Unless exlicitly stated otherwise, the letter $P$ will stand for  such polynomials, and $q$ for such forms. For a given polynomial $P$ the letter $Q$ will denote the corresponding form obtained by setting $d,e,f=0$. 
   We  reserve the letter $k$ for integer values $P(m,n)$ takes on integer inputs $m,n$. The pair $(m,n)$ is called a representation of $k$ by the polynomial if $P(m,n)=k$. When we want to consider our polynomials on real numbers we will prefer writing $P(x,y)$.    For a set $E$, we  let $\# E$ denote its cardinality, and $|E|$ its Lebesgue measure. 
   
   For a form $q$, if $\Delta<0$,  which implies $ac>0$, the form is called definite as it is   nonnegative or nonpositive on all real entries. This is clear by the identity
   \[q(x,y)=ax^2+bxy+cy^2=a\Big[ \big(x+\frac{b}{2a}y     \big)^2 -\frac{\Delta y^2}{4a^2}   \Big] .   \]
    If $a,c>0$, being always nonnegative, the form is called  positive definite whereas if $a,c<0$, it is called  negative definite analogously. When $\Delta>0$ the form  is called indefinite.

     In recent joint work with E. Sert \cite{ts}, we proved that with $q(m,n)$ an integral binary quadratic form with negative or positive nonsquare discriminant, $\mathcal{I}_{\lambda,q}$ is a bounded operator on $l^p(\mathbb{Z}), \  1\leq  p<\infty$  for $\lambda >1-p^{-1}$.  Such  forms are of course neither translation invariant nor quasi translation invariant.   We further showed that for $p=1$ the denominator $|m|^{\lambda}$ cannot be replaced by $\log^r (1+|m|)$, and for $1<p<\infty$ we cannot take $\lambda =1-p^{-1}$.
 The main framework set out in  that  article to prove these results will  be in use in this work as well, and we present it below in order to make discussion of results and ideas  of both articles possible. Further familiarity with that article is not necessary,  but could be helpful.

 In our framework the  $p=1$  emerges as the guiding case, we therefore focus our discussion on it.  The cases $p>1$ follow from the same arguments applied after the Hölder inequality.  
 Let $P$ be an  arbitrary  bivariate quadratic polynomial of integral coefficients. We have
 \begin{equation*}\label{}
 \|\mathcal{I}_{\lambda}f\|_{l^1(\mathbb{Z})}= \sum_{n\in \mathbb{Z} }\Big|  \sum_{m\in \mathbb{Z}_*} \frac{f(P(m,n))}{|m|^{\lambda }}     \Big| \leq \sum_{(m,n)\in \mathbb{Z}_*\times \mathbb{Z}} \frac{|f(P(m,n))|}{|m|^{\lambda }}.     
 \end{equation*}
 We define    for each $k\in \mathbb{Z}$ the sets 
$
 A_{k}:=\{(m,n)\in \mathbb{Z}_* \times \mathbb{Z}: P(m,n)=k\}
$
that  form a partition of $\mathbb{Z}_* \times \mathbb{Z}$.  Therefore 
 \begin{equation*}\label{}
= \sum_{k\in \mathbb{Z} }|f(k)| \Big[ \sum_{(m,n)\in A_k} \frac{1}{|m|^{\lambda }}\Big].     
 \end{equation*}
As  $f\in l^1(\mathbb{Z})$, showing that 
  \begin{equation}\label{1.2}
 \sum_{(m,n)\in A_k} \frac{1}{|m|^{\lambda }}   
  \end{equation}
 is bounded by  a constant $C$ independent of $k$ would yield 
 $
  \|\mathcal{I}_{\lambda}f\|_{l^1(\mathbb{Z})}  
 \leq C
 \|f\|_{l^1(\mathbb{Z})}. 
 $
 Thus our framework  reduces   the problem  to understanding the quantity \eqref{1.2}, and this is related to the number and distribution of   representations of $k$ by the phase polynomial. 
 
 Binary quadratic forms are the simplest  bivariate   quadratic  polynomials, and further,
 representations by the latter can be investigated by transforming them into the  former by simple algebraic operations. Therefore it is most reasonable to initiate the study of   bivariate   quadratic   phase polynomials with binary quadratic forms.  Behavior of a quadratic form mostly depend on the sign of  its discriminant, and whether this discriminant is a full square, so we classify these forms accordingly. The negative discriminant and the positive nonsquare discriminant cases  comprise the content of our work \cite{ts}.
 
In the negative discriminant case we have at most  $C_{\Delta, \varepsilon}k^{\varepsilon}$  representations $(m,n)$ for an integer $k$, and of these only 4 can have $|m|\leq |k|^{1/4}/\sqrt{|\Delta|}$, so \eqref{1.2} is bounded for any $\lambda>0$. For  the  positive nonsquare discriminant case, we have infinitely many such representations, but these are generated from powers of  a fixed matrix,  which after an appropriate partition into $C_{\Delta, \varepsilon}k^{\varepsilon}$ subsets, allows us to demonstrate that  members of each subset   grow exponentially.  So we can treat this case as if there are at most $C_{\Delta, \varepsilon}k^{\varepsilon}$  representations. Also,  again only  4 representations can have $|m|\leq |k|^{1/4}/\sqrt{|\Delta|}$. Combining these we   bound  \eqref{1.2}  for any $\lambda >0$.

If the form  has a positive square discriminant, it factorizes into two distinct linear terms of integer coefficients,  and thus $0$  have  representations the first entries of which form an arithmetic progression. For example, letting $q(m,n)=m^2-n^2$, we see that $(m,m)$ represents $0$ for any integer $m$. So bounding $\eqref{1.2}$ is not possible even with $\lambda=1$, and nontrivial estimates are not possible for  $l^1(\mathbb{Z})$. But we must also notice that $0$ is the only number with this type of behavior, and any nonzero integer $k$ have only finitely many representations arising from its divisors. As the number of divisors is bounded by $C_{\varepsilon}k^{\varepsilon}$, if we can isolate the representations of $0$, and make sure   for the representations $(m,n)$ of nonzero $k$  that the values $|m|$ are away from $0$, we may have boundedness results for $l^p(\mathbb{Z}), \ p>1$, by bringing the exponent $p$ into the summation  \eqref{1.1}.  With these observations we prove our first theorem.

 \begin{theorem}
 Let $f\in l^p(\mathbb{Z})$  where $1<p<\infty$. Let $q$ be an    integral binary quadratic form with  positive square  discriminant $\Delta=\delta^2, \ \delta>0$, and $a,c\neq 0$. Then the  operator 
\eqref{1.1}
 satisfies
 \begin{equation*}
 \|\mathcal{I}_{\lambda}f\|_p \leq C_{\lambda,p,\Delta}\|f\|_p
 \end{equation*}
 for   $\lambda >\max\{1-p^{-1}, p^{-1}\}$. 
 \end{theorem}

 Here the condition $a\neq 0$ is clearly  necessary, for otherwise  for a function $f$ nonzero at the origin the  sum in \eqref{1.1} becomes infinite when $n=0$.  On the other hand 
it may be possible to obtain the same estimates as in this theorem with $c=0$, but this requires new ideas. For in this case solutions of  $am^2+bmn=k, \ k\neq 0$ lie on hyperbolas  for which the $y$-axis  is an asymptote.  To see clearly how this leads to difficulties, take  $k$ to be the product of first $j$ primes, and $a=b=1$. Then  the quantity \eqref{1.2} cannot be less than the sum of inverses of first $j$ primes, and hence not bounded by a  constant independent of $k$.  So 0 is not the only problematic value in this case. We find investigation of this case to be very worthwhile, as it may to lead to new connections to arithmetic. We may conduct such an investigation in a future article.

Our second  theorem represents a vast generalization of our work on binary quadratic forms to  bivariate  quadratic polynomials. The main idea is to use algebraic operations to reduce representation by a polynomial $P$ to  representation by the corresponding form $Q$, and  then  use a  decomposition and estimates obtained in \cite{ts} on sums of type \eqref{1.2}.

 \begin{theorem}
 Let $f\in l^p(\mathbb{Z})$  where $1\leq p<\infty$. Let $P$ be a 
 be an integral  bivariate   quadratic  polynomial with nonzero discriminant.  Let $\Gamma:=\Delta^{-1}Q(e,-d)+f$. Then the  operator \eqref{1.1}
 satisfies
 \begin{equation}\label{1.3}
 \|\mathcal{I}_{\lambda}f\|_p \leq C_{\lambda,p, \Delta,\alpha}\|f\|_p
 \end{equation}
 for $\lambda >1-p^{-1}$ if $\Delta$ is negative or  positive nonsquare,  and for
  \begin{equation*}\label{}
  \lambda > \begin{cases} 1-p^{-1} &\textnormal{if}  \ \    \{(m,n)\in \mathbb{Z}^2: \  P(m,n)=\Gamma \}=\emptyset      \\
   \max\{p^{-1},1-p^{-1}\} &\textnormal{else}   \end{cases}
   \end{equation*}
 if $\Delta$ is a positive square.
 \end{theorem}
 
 Thus remarkably  there are polynomials of  positive square discriminant  that satisfy estimates much better than those satisfied by  forms of positive square discriminant, in particular they have estimates for $p=1$.
 Indeed the part of the theorem  regarding polynomials of positive square discriminant will be made more clear  by expressing the solvability of $P(m,n)=\Gamma $ in terms of coefficients of  $P$.  As this requires yet more notation we defer  it to   section 3.

 The  generalization has one weak point, which is that now our constants, in addition to $\lambda,p,\Delta$, depend on $\alpha.$ 
This is an issue related to  a cycle of  very difficult conjectures  in number theory regarding the concentration of lattice points on short arcs of conics. Here we state the  conjectures  most relevant to us. For other conjectures in this circle and relations between them as well as their connections to other outstanding problems in analysis such as sum-product sets,  exponential sums, squares in arithmetic progressions see \cite{cg, cu}. For results on extensions of these conjectures to higher dimensions, which turn out to be more tractable as with many other problems regarding lattice points on surfaces, and their applications to the eigenfunctions of the Laplacian on torii see   \cite{br1,br2}.

\begin{conjecture} Let $N\in \mathbb{N}$ and $0<\eta<1/2$. Then  the set
 \begin{equation*}
 \{(m,n)\in \mathbb{Z}^2:  m^2+n^2=N,    \ \    |n|<N^{\eta}\}
 \end{equation*}
 has cardinality bounded by a constant  $C_{\eta}$ independent of $N$.
 \end{conjecture} 
 This is clear for $\eta\leq 1/4$, but beyond this only logarithmic improvements for $N$  a square plus a much smaller square are known by the work of Chan \cite{thc1,thc2}.  There is also a simple argument due to  Bourgain and Rudnick, see \cite{br2}, yielding   the conjecture  with the possible exception of a sparse set of $N$. Below we delve deeper into this conjecture, but now to show its relation to Theorem 2 we explicitly compute for the phase polynomials $P_j(m,n)=m^2+n^2+2jm$ and  $f$  the point mass at $0$  
 
\begin{equation*}\label{1.4}
\begin{aligned}
\|\mathcal{I}_{1, P_j}f\|_1= \sum_{\substack{(m,n)\in\mathbb{Z}_*\times \mathbb{Z}\\  P_j(m,n)=0}}\frac{1}{|m|^{\lambda}} = \sum_{\substack{(m,n)\in\mathbb{Z}_*\times \mathbb{Z}\\  (m+j)^2+n^2=j^2}}\frac{1}{|m|^{\lambda}} =  \sum_{\substack{(m,n)\in\mathbb{Z}-\{j\}\times \mathbb{Z}\\  m^2+n^2=j^2}}\frac{1}{|m-j|^{\lambda}}.
\end{aligned}
\end{equation*}
 For the polynomials $P_j$ the value $\alpha=2cd-be=4j$ obviously depends on $j$ but the discriminant $\Delta=4$ is independent of it. Therefore  we are required to bound the last sum above
  independently of $j$ in order to obtain estimates independent of $j$ for $\|\mathcal{I}_{\lambda, P_j}f\|_1$.  The main contribution to that sum  comes from $(m,n)$ with $m$ close to $j$, and these are exactly the points with $|n|$ small. Therefore  Conjecture 1 for any $\eta>1/4$ immediately implies  the  boundedness of that sum independent of $j$ for any $\lambda$.   So we may wiev this problem as a weaker form of Conjecture 1: while Conjecture 1  claims that lattice points  $(m,n)$ with $|n|$ small are finite,   our problem  requires such points to be merely sparse.

  Applying the large sieve via quadratic residues and the prime number theorem in arithmetic progressions,  we solve this problem for  $\lambda>1/2$. Indeed we  consider not just circles but a rather general class of conics that suffices to handle all polynomials of negative or positive nonsquare discriminant case. 
   
   \begin{theorem}
   Let $a$ be an integer such that $-a$ is a nonsquare, and $\lambda>1/2$. Let $\tau\in\mathbb{Z}$. Then the sum  
   \begin{equation}\label{1.5}
     \sum_{\substack{(m,n)\in\mathbb{Z}-\{\tau\}\times \mathbb{Z}\\  am^2+n^2=N}}\frac{1}{|m-\tau|^{\lambda}}
     \end{equation}
    is bounded by a constant independent of $N,\tau$. 
     \end{theorem}
  Applying this result immediately yields
  \begin{theorem}
   Let $f\in l^p(\mathbb{Z})$  where $1\leq p< \infty$. Let $P$
   be an integral  bivariate   quadratic  polynomial with negative or positive nonsquare discriminant. Then   for $\lambda >1-(2p)^{-1}$  the  operator \eqref{1.1}
   satisfies
   \begin{equation*}\label{}
   \|\mathcal{I}_{\lambda}f\|_p \leq C_{\lambda,p,\Delta}\|f\|_p.
   \end{equation*}
   \end{theorem}
  
For a similar result on polynomials of positive square discriminant we would need to cover the case $-a$ a square,  but   unfortunately our method does not extend there.  Such a result would constitute a weaker version of  a well known analogue of Conjecture 1   posed by I. Ruzsa. 
   \begin{conjecture} Let $N\in \mathbb{N}$ and $0<\eta<1/2$. Then  the set
    \begin{equation*}
    \{(m,n)\in \mathbb{Z}^2:  mn=N,    \ \      |n-N^{1/2}|<N^{\eta}\}
    \end{equation*}
    has cardinality bounded by a constant  $C_{\eta}$ independent of $N$.
    \end{conjecture} 
   This conjecture too is trivial for $\eta\leq 1/4$, and is not known for any larger $\eta$. There is logarithmic improvement by Chan in \cite{thc1,thc2}  for   $N$  a square minus a much smaller square, and an on average version of  the question was   studied in \cite{cu}.

 These theorems   establish a very strong connection between discrete fractional integrals and concentration of lattice point on short arcs of conics. This latter topic is connected to diophantine approximation, as can be seen from the works  \cite{br1,thc1,thc2,jt}.  Diophantine approximation  is deeply interrelated with  the existence and boundedness of solutions of certain  diophantine equations, such as  Pell and Thue equations.  Algebraic numbers lack good rational
 approximation, and this fact is encapsulated by the two main theorems of diophantine approximation, that is Roth's theorem \cite{r}, and Schmidt's theorem \cite{sch}.  These theorems are sharp but ineffective, and over the last 50 years tremendous effort has been spent on proving effective versions of these theorems, and using them to study questions regarding simultaneous Pell equations and  diophantine $m$-tuples. For  a  starting point to this literature we recommend the articles \cite{chud,ric}. Despite the vast literature we are still far from strong effective results.

 Our next theorem and its proof highlight the  connections between lattice point problems and diophantine approximation most clearly.   Specifically we will use Schmidt's theorem on simultaneous approximation \cite{sch} to obtain a finiteness result for lattice points on circles.
 
 \begin{theorem}
 Let $0< h_1<h_2<\ldots <h_l$ be integers with $l\geq 5,$ and let 
 \[S:=\{N\in \mathbb{N}:  \  N=R^2+r, \ \     R\in \mathbb{N},   \ \  |r|\leq R^{\frac{1}{2}-\rho}                           \}     \]
  where $0<\rho\leq 1/2$  and $2\rho(l-1)>1.$
Then the subset $S'\subseteq S$ of all $N$ such that for each $1\leq i\leq l$, there exist  $n_i\in \mathbb{N}$ with $(R-h_i)^2+n_i^2=N$ is finite.
 \end{theorem}  
 This theorem does not lead to any new results on Conjecture 1 for any values of $N$, but its proof makes it plain that this is because Schmidt's theorem is ineffective, which forces us to fix $h_i$ beforehand. If we knew the constant of that theorem, and if it were of appropriate size,  we would obtain Conjecture 1 for some  $\eta>1/4$ for values in $S$ with some $\rho$.  On the other hand,  it can be viewed as progress towards the study of patterns of lattice points on conics.  This study  can be conducted   on the plane, or via projections on the axes.  Within this latter framework, in the particular case of  circles,  we can rigorously  formulate the  problem   as follows.  Let $0<h_1<h_2<\ldots<h_l$ be fixed integers. Consider the set of $N\in \mathbb{N}$ such that for each $0\leq i\leq l$ we have lattice points $(m_i,n_i)$ of nonnegative coordinates satisfying $m_i^2+n_i^2=N$ and $m_0-m_i=h_i$. We would like to know   whether this set  is finite. For $l$  large this would follow from a well known conjecture,  known,  see \cite{cg},  to be equivalent to Conjecture 1.
 \begin{conjecture}
  On the circle centered at the origin with radius $\sqrt{N}, \ N\in \mathbb{N}$, an arc of length $N^{\eta}, \ \eta<1/2$ can contain at most $C_{\eta}$  lattice points, independent of $N$. 
 \end{conjecture}
This  is known for $\eta <1/4$ by the work of Cilleruelo and Cordoba \cite{cc}.  An arc containing all of $(m_i,n_i), \ 0 \leq i \leq l$ has a length not exceeding $3h_lN^{1/4}$. Therefore if this conjecture holds for any $\eta>1/4$ with $C_{\eta}\leq l$,
  this pattern cannot occur infinitely often. Even the result of Cilleruelo and Cordoba  is sufficient to see  that to repeat infinitely often these  patterns must lie close to the right end of the interval $[0,\sqrt{N}]$, for the arclength requirement can only be satisfied there.  Lastly it is known that certain patterns do repeat infinitely often, e.g  $0<1<2$ with lattice points $(4j^3-1,2j^2+2j),(4j^3, 2j^2+1), (4j^3+1,2j^2-2j)$.

 The proof of  Theorem 5 relies on obtaining simultaneous Pell equations, as does Chan \cite{thc1,thc2},  but we view them as hyperbolas with asymptotes of algebraic slope. Points on hyperbolas yield a very good simultaneous approximation to these algebraic slopes and  this reveals an immediate opportunity to apply   Schmidt's theorem.  Chan on the other hand applies  Turk's effective  bounds on solutions of simultaneous Pell equations.   We remark that the ideas used to prove  Theorem 5 can also be used to connect the lattice point problems to  uniform distribution modulo 1, an  area itself very closely connected to diophantine appoximation.  For the  very good rational approximation yielded by hyperbolas  also violates uniform distribution. But as the results so obtained are weaker  than Theorem 5 we will not state them. We further remark that in order to bound the number of squares in arithmetic progressions  the articles  \cite{bgp,bz} rely on obtaining equations of elliptic curves by  eliminating variables, much like we do for   Theorem  5. It may be possible to use the  methods there   in conjuction with  our methods  to obtain results similar to  Theorem 3 for $N$ a full square and $\lambda >1/2$. But as this result would be weaker than Theorem 3 we will not explore this possibility here.

 Our last theorem  builds upon Chan's ideas to improve  his theorems. We do this via a simpler but more efficient way of dealing with the  exceptions to  applicability of  effective results on the size of solutions of simultaneous Pell equations.  Also instead of Turk's result, we apply  a recent theorem of Bugeaud \cite{yb}  that improves upon it. 
 \begin{theorem}
Let $\kappa=1/2-\varepsilon$ where $0<\varepsilon<10^{-2}$, and  
\[E:=\{N\in \mathbb{N}:  \  N=R^2+r, \ \     R\in \mathbb{N},   \ \  |r|\leq 16e^{2\log^\kappa N }                           \}.     \]
\[F:=\{N\in \mathbb{N}:  \  N=R^2+r, \ \     R\in \mathbb{N},   \ \  |r|\leq 4e^{2\log^\kappa N }                           \}.     \]
Let $N$ be large enough, for example $\log N \geq K^{\varepsilon^{-3}},$    where  $K:=\max\{C,3\}$ and  $C$ is the absolute constant that appears in Bugeaud's theorem.  For lattice points on circles when    $N\in E$ we have
 \begin{equation}\label{1.6}
 \# \{(m,n)\in \mathbb{Z}^2:  m^2+n^2=N,    \ \    |n|\leq 6N^{1/4}\log^{\kappa/4}N \}\leq 20.
 \end{equation}
 For lattice points on hyperbolas when  $N\in E$ we have
  \begin{equation}\label{1.7}
  \# \{(m,n)\in \mathbb{Z}^2:  m^2-n^2=N,    \ \    |n|\leq 6N^{1/4}\log^{\kappa/4}N \}\leq 20.
  \end{equation}
For divisors of  $N\in F$  we have
 \begin{equation}\label{1.8}
   \# \{(m,n)\in \mathbb{Z}^2:  mn=N,    \ \   |n-N^{1/2}|\leq 2N^{1/4}\log^{\kappa/4}N \}\leq 10.  
  \end{equation}

 \end{theorem}

 As is clear to the careful reader we have excluded  zero discriminant  bivariate    quadratic polynomials from our analysis.   This case seems to have three different boundedness ranges.   For  polynomials reducable to squares of  linear polynomials by completion of squares, the boundedness range is given by the Hardy-Littlewood-Sobolev theorem. For polynomials reducable to  the case $n-m^2$, the exponents of this particular polynomial, obtained as a result of   such works as \cite{ao,imsw,iw,ob,sw,sw2} are valid. The remaining polynomials are reducable to the case $m+n^2$, and for these we are able to attain the sharp exponents.  As adding all these to this article would make it somewhat cumbersome, they will be presented  in a future work.

 We remark that  for $p=\infty$ only trivial estimates, that is estimates with  $\lambda>1$, exist, therefore we do not consider this case. Also we do not prove off-diagonal estimates as  no significant extension  of those that immediately follow from our diagonal estimates is possible. To observe this let $P(m,n)=m^2+n^2$ and consider the $(p,q), \ p\neq q$ estimate
 \[   \|\mathcal{I}_{\lambda}f\|_q   \leq C_{\lambda,p,q} \|f\|_p.  \]
  It is not  possible,  by  raising $\lambda$ if necessary, to prove an estimate with $p>q$.  We see this  by just taking for  positive and small $\varepsilon$
 \[  f(k)=    \begin{cases}   j^{-p^{-1}-\varepsilon} \ \ \  &\text{if}   \ \ \  k=j^2+1, \ j\in \mathbb{N}  \\  0  \ \ \  &\text{else,} \end{cases}          \] 
 and calculating 
 \begin{equation*}
 \|\mathcal{I}_{\lambda}f\|_q ^q= \sum_{n\in \mathbb{Z}}  \Big[\sum_{m\in \mathbb{Z_*}} \frac{f(m^2+n^2)}{|m|^{\lambda}}  \Big]^q\geq \sum_{n\in \mathbb{N}}   f^q(1+n^2)\geq \sum_{n\in \mathbb{N} } n^{-(p^{-1}+\varepsilon)q}=\infty.
 \end{equation*}
 On the other hand, estimates with $p<q$, with the same $\lambda$,  obviously follow from the case $p=q$.  So here the only interesting question is whether we can lower $\lambda$ as we raise $q.$ This is not possible either,  if we estimate the same example in a different way:
 \begin{equation*}
  \|\mathcal{I}_{\lambda}f\|_q ^q= \sum_{n\in \mathbb{Z}}  \Big[\sum_{m\in \mathbb{Z_*}} \frac{f(m^2+n^2)}{|m|^{\lambda}}  \Big]^q\geq \Big [\sum_{m\in \mathbb{N}}   \frac{f(m^2+1)}{|m|^{\lambda}}\Big ]^q \geq \Big [\sum_{m\in \mathbb{N} } m^{-p^{-1}-\lambda-\varepsilon}\Big ]^q,
  \end{equation*}
 which means $\lambda\geq 1-p^{-1}$,  i.e. essentially  the same condition as in the $p=q$ case.
Hence focusing on   diagonal estimates is not  restrictive at all.

 The contents of the rest  of the article is as follows.  The next section presents, after exhibiting  arithmetic and analytic properties of quadratic forms of nonzero square discriminant in preparation, the proof of  Theorem 1.  In section  3   we reduce  representation of an integer by a polynomial via  translation and dilation to representation by the corresponding form, and use this to prove Theorem 2.  
In section 4, after we review the large sieve,  prime number theorem  in arithmetic progressions, and quadratic reciprocity we prove Theorem 3. Inserting this into the proof of Theorem 2 gives Theorem 4. The last section first proves Theorem 5 after reviewing Schmidt's theorem and a result of Besicovitch needed to implement it. Then we describe Bugeaud's recent theorem and  use it to prove Theorem 6, after which we describe how our Theorem 6 relates to Chan's work.


\section{Binary Quadratic Forms of  Positive Square Discriminant}

The aim  of this section is to prove Theorem 1. As briefly discussed in the  introduction, the proof rests on three  ingredients:  that the cardinality of $A_k, \ k\neq 0$ is small,  that for $(m,n)\in A_k, \  k\neq 0$ the value $|m|$ is distant from 0, and that we must  isolate the  representations of $0$. The first two will be achieved before the proof proper,   whereas the  last will be carried out at the beginning of the proof.
We now investigate our form   using algebra and arithmetic to obtain the first ingredient. Then we will use geometry and analysis to obtain the second.  With the first two ingredients at hand, we will be ready for presenting the proof of Theorem 1.

The key property of forms of positive square discriminant is that they factor into two disctinct linear factors of integer coefficients.  Our form is $q(m,n)= am^2+bmn+cn^2$ with $a,c\neq 0 $, and $\Delta=\delta^2$ with $ \delta \in \mathbb{N}.$  As can immediately be verified by multiplication
\begin{equation}\label{2.1}
am^2+bmn+cn^2=a\big(m+\frac{b-\delta}{2a}n\big)\big(m+\frac{b+\delta}{2a}n\big).
\end{equation}
As $(b+\delta)+(b-\delta)=2b$, they must be of the same parity, and as $(b-\delta)(b+\delta)=4ac$ they are both even.  Thus $(b\pm \delta)/2$ are integers, and we define the integers
 \[a_1:=\gcd\big(\frac{b-\delta}{2},a\big)  \ \  \ \ \ \ a_2:=\frac{a}{a_1}, \ \   \ \ \  c_1:=\frac{b-\delta}{2a_1}.  \]  
Since $\gcd(a_2,c_1)=1$, and
\[c_1\frac{b+\delta}{2}=a_2c,\]
$a_2$ divides $(b+\delta)/2$, and we define $c_2$ to be the  result of this division. Hence  we have $c=c_1c_2$, and the factorization in \eqref{2.1} becomes
\begin{equation}\label{2.2}
=a_1a_2\big(m+\frac{c_1}{a_2}n\big)\big(m+\frac{c_2}{a_1}n\big)=(a_2m+c_1n)(a_1m+c_2n).
\end{equation}

 From here it is easy to deduce information regarding representations. Let $g_1:=\gcd(a_2,c_1)$ and $g_2:=\gcd(a_1,c_2)$. The representations of $0$ are $j(c_1/g_1,-a_2/g_1),   \  j\in \mathbb{Z}$  and        $  j(c_2/g_2,-a_1/g_2),   \   j\in \mathbb{Z}$. As for $k\neq 0$, the map $(m,n)\mapsto (a_2m+c_1n,a_1m+c_2n)$ on $A_k$ is an injective mapping  into the set of elements $(u, k/u)$ where $u\in \mathbb{Z}_*$ divides $k$. Hence the cardinality of $A_k$ cannot exceed the number of divisors of $k$, and as is well known,  this is bounded by $C_\varepsilon k^{\varepsilon}$ for every $\varepsilon >0$.
 
 From analytic and geometric points of view our  forms   are very much like forms of positive nonsquare discriminant.    We investigate the set   $\{(x,y)\in  \mathbb{R}^2: q(x,y)= z \} $ for every $z\in \mathbb{R}$.  We assume that for our form $c>0$, the case $c<0$ immediately follows.  For $z=0$ the factorization above gives  two distinct lines 
\begin{equation}\label{2.3}
 a_2x+c_1y=0,    \quad \quad  \quad \quad  a_1x+c_2y=0,
 \end{equation}
and as the coeffiecients  are nonzero these lines are neither vertical nor horizontal. 
When $z>0$, the set is  a hyperbola centered at the origin with the lines in \eqref{2.3} as asymptotes. The graphs of 
\begin{equation}\label{2.4}
y=f_1(x)=\frac{-bx+\sqrt{\Delta x^2+4cz}}{2c},  \quad \quad \quad y=f_2(x)=\frac{-bx-\sqrt{\Delta x^2 +4cz}}{2c},  
\end{equation}
give the two  components of the hyperbola,  with $f_1$ lying above both asymptotes and $f_2$  lying below both of them.   
With  $z<0$ we obtain the  conjugate of the hyperbola we would have  for $-z$. Its two  components lie between the asymptotes, and points $(x,y)$ on it satisfy $x^2\geq -4cz/\Delta.$

\begin{lemma}
Let $q(x,y)= ax^2+bxy+cy^2$ be an integral binary quadratic form with $a,c\neq 0 $, and $\Delta=\delta^2$ for a natural number  $\delta$. Let $k\neq 0$ be an integer. Then $q(x,y)=k$ has at most 4 solutions $(x,y)\in \mathbb{Z}^2$ satisfying $ |x| \leq |k|^{1/4}/\delta$.
\end{lemma}
\begin{proof} 
 
 We assume $c>0$ as the case $c<0$ follows from this by considering $-q$ and $-k$.
If $k<0$, as we remarked any solution to 
 $q(x,y)=k$ satisfies $x^2\geq -4ck/\Delta$, that is, there are no solutions of the type asked for in the lemma.  So it remains to consider positive integers $k.$
 
  The solutions we are looking for lie on the  graphs of the functions $f_1,f_2$ in \eqref{2.4}. As these graphs are disjoint
any of these solutions can lie on only one of these graphs. 
The lines 
\[l_1(x)=-\frac{b}{2c}x+\sqrt{\frac{k}{c}}, \quad \quad l_2(x)=-\frac{b}{2c}x-\sqrt{\frac{k}{c}} \]
are respectively tangent to $f_1,f_2$ at $x=0$.
   We will  prove that  $f_1,f_2$  stay very close to these lines for $|x| \leq |k|^{1/4}/\delta$ . The differences   $|f_i(x)-l_i(x)|, \ i=1,2$ for such $x$  are bounded   by
\begin{equation*}
\begin{aligned}
\frac{\sqrt{\Delta x^2+4ck}}{2c}-\sqrt{\frac{k}{c}}&=  \Big(\frac{{\Delta x^2}}{4c^2} +\frac{k}{c}\Big)^{1/2}-\Big(\frac{k}{c}\Big)^{1/2} \\ &=\frac{{\Delta x^2}}{4c^2} \cdot \Big[\Big(\frac{{\Delta x^2}}{4c^2} +\frac{k}{c}\Big)^{1/2}+  \Big(\frac{k}{c}\Big)^{1/2} \Big]^{-1} \\ &\leq \frac{{\Delta x^2}}{4c^2} \cdot \Big[ {2} \Big(\frac{k}{c}\Big)^{1/2} \Big]^{-1} \\ & \leq \frac{1}{8c^{3/2}}.
\end{aligned}
\end{equation*}
Therefore our solutions satisfying $y=f_i(x)$ lie inside the   set
$S_i:=\{(x,y)\in \mathbb{R}^2:  |y-l_i(x)|  \leq {1}/{8c^{3/2}}    \} $
for $i=1,2.$ These two sets are clearly disjoint.

Yet, if $(m,n)\in \mathbb{Z}^2$, then $2cn+bm=j\in \mathbb{Z} $, which means 
\[n=-\frac{b}{2c}m+\frac{j}{2c}.\]
Therefore every element $(m,n)\in \mathbb{Z}^2$ lies on exactly one of the collection of parallel lines 
\[\Big\{(x,y)\in \mathbb{R}^2:  y=-\frac{b}{2c}x+\frac{j}{2c}\Big\}.\]
 But the sets $S_1,S_2$ each can contain at most one line from this collection. As a  line can  intersect $q(x,y)=k$ at most twice,  we have at most 4 solutions.

\end{proof}

We are now ready to present the proof of Theorem 1.   After removing the representations of $0$, we will apply the Hölder inequality to reduce to a sum of the type \eqref{1.2}. Then mobilizing what we uncovered from our investigations in this section, we will obtain the desired conclusion.

\begin{proof}
We start with separating the representations of $0$ and applying the Hölder inequality
\begin{equation*}
\begin{aligned}
\|\mathcal{I}_{\lambda}f\|_{p}^p&\leq \sum_{n\in \mathbb{Z}}\Big[\sum_{\substack{m\in \mathbb{Z}_*\\ q(m,n)=0}} \frac{|f(0)|}{|m|^{\lambda }} +\sum_{\substack{m\in \mathbb{Z}_*\\ q(m,n)\neq 0}} \frac{|f(q(m,n))|}{|m|^{\lambda }}  \Big]^p \\ &\leq 2^{p-1}|f(0)|^p \sum_{n\in \mathbb{Z}}\Big[\sum_{\substack{m\in \mathbb{Z}_*\\ q(m,n)=0}} \frac{1}{|m|^{\lambda }} \Big]^p +2^{p-1} \sum_{n\in \mathbb{Z}}\Big[\sum_{\substack{m\in \mathbb{Z}_*\\ q(m,n)\neq 0}} \frac{|f(q(m,n))|}{|m|^{\lambda }}\Big]^p  .
\end{aligned}
\end{equation*}
 We will handle the first sum now.  As we assumed $a,c\neq 0$, there can be at most two solutions  $(m,n)\in \mathbb{Z}_*\times \mathbb{Z}$ to $q(m,n)=0$  when one of the entries is fixed,  therefore applying the Hölder inequality, the sum is bounded by
 \[2^{p-1}\sum_{n\in \mathbb{Z}}\sum_{\substack{m\in \mathbb{Z}_*\\ q(m,n)=0}} \frac{1}{|m|^{\lambda p }}  \leq 2^{p-1}\sum_{m\in \mathbb{Z}_*}\sum_{\substack{n\in \mathbb{Z}\\ q(m,n)=0}} \frac{1}{|m|^{\lambda p }} \leq 2^{p}\sum_{m\in \mathbb{Z}_*} \frac{1}{|m|^{\lambda p }}\leq C_{\lambda,p}.  \]

We  turn to the second sum. Let  $p'$ denote the dual exponent of  $p$, and $\lambda':=\lambda-(1-p^{-1}).$ Then applying the Hölder inequality  the second sum is bounded by  
\begin{equation*}
\begin{aligned}
\sum_{n\in \mathbb{Z}}\Big[\sum_{\substack{m\in \mathbb{Z}_*\\ q(m,n)\neq 0}} \frac{|f(q(m,n))|^p}{|m|^{(\lambda'/2)p }}\Big]  \Big[\sum_{\substack{m\in \mathbb{Z}_*\\ q(m,n)\neq 0}} \frac{1}{|m|^{(\lambda-\lambda'/2)p'}}\Big]^{p-1}.
\end{aligned}
\end{equation*}
Since $(\lambda-\lambda'/2)p'>1$, this in turn is bounded by 
\begin{equation*}
\begin{aligned}
C_{\lambda,p}\sum_{\substack{(m,n)\in   \mathbb{Z}_*\times \mathbb{Z} \\ q(m,n)\neq 0    }} \frac{|f(q(m,n))|^p}{|m|^{\lambda'p/2 }}= C_{\lambda,p} \sum_{k\in \mathbb{Z}_*}|f(k)|^p\sum_{(m,n)\in A_k} \frac{1}{|m|^{\lambda'p/2 }}    .  
\end{aligned}
\end{equation*}
Thus we need to bound the inner sum. By Lemma 1 we have at most  $4$ solutions in $A_k$ with $|m|\leq |k|^{1/4}/\delta$, and the cardinality of $A_k$ is bounded by $C_{\varepsilon}|k|^{\varepsilon}$. Choosing $\varepsilon=\lambda'p/10$,  we conclude the proof with  
\begin{equation}\label{2.5}
\begin{aligned}
\sum_{(m,n)\in A_k} \frac{1}{|m|^{\lambda'p/2 }}\leq 4+ C_{\lambda,p}\delta^{\lambda'p/2}|k|^{\lambda'p/10-\lambda'p/8}\leq C_{\lambda,p,\Delta}.
\end{aligned}
\end{equation}

\end{proof}


\section{Extension to Polynomials}

In this section we prove Theorem 2.  We rely on  algebraic operations to reduce to the case of binary quadratic forms, and once  there  use the bounds in \cite{ts} for sums of the type \eqref{1.2}.  We will also clearly observe where and how  the dependence on $\alpha$  arises. Thus once we prove Theorem 3 in the next section, Theorem 4 will easily follow.

Before the proof proper, we demonstrate the reduction idea for any polynomial
  $P$ of nonzero discriminant.  Let $Q$ be the corresponding form.   Let $m=u+r, \ n=v+s$. Then  we have
\begin{equation*}
\begin{aligned}
P(m,n)=P(u+r,v+s)&=a[u^2+2ur+r^2]+b[uv+us+vr+rs]\\ &+c[v^2+2vs+s^2]+d[u+r]+e[v+s]+f \\ &=Q(u,v)+u[2ar+bs+d]+v[2cs+br+e]+P(r,s).
\end{aligned}
\end{equation*}
So to annihilate the first order terms we need 
\begin{equation}\label{y3.2} 
 \begin{bmatrix}
2a  & b \\ b &  2c
\end{bmatrix}  
     \begin{bmatrix}
     r  \\   s
     \end{bmatrix} =   \begin{bmatrix}
          -d  \\   -e
          \end{bmatrix}.       \end{equation}
As the discriminant is nonzero, the unique solution pair  is  $r=\alpha/\Delta, \ s=\beta/\Delta$. We also observe that
\begin{equation*}
\begin{aligned}
P(r,s)=\frac{1}{2}\Big[r(2ar+bs+d)+ s(2cs+br+e) +dr+es \Big]+f &=\frac{dr+es}{2}+f=\Gamma
\end{aligned}
\end{equation*}
Then $P(m,n)=k$ if and only if
$Q(m-\alpha/\Delta,n-\beta/\Delta)=k-\Gamma .$
To deploy the theory of representation of integers by quadratic forms we multiply both sides by $\Delta^2$, and  turn  the  variables of this last equality into integers. Let $m':=\Delta m-\alpha$ and $n':=\Delta n-\beta$, and also $k'=\Delta^2(k-\Gamma).$     The map $(m,n,k)\mapsto (m',n',k')$ clearly is injective.  Thus  $P(m,n)=k$  if and only if $Q(m',n')=k'$.

\begin{proof}

We start with polynomials of negative or positive nonsquare discriminant, and first investigate $p=1$ case. The general case will follow from similar arguments after applying the Hölder inequality.  As  made clear in the introduction we are to  bound  
\eqref{1.2} uniformly in $k.$  It is bounded by

\begin{equation}\label{y3.1}
\begin{aligned}
\sum_{\substack{Q(m,n)=k'  \\ m\neq -\alpha}} \frac{|\Delta|^{\lambda}}{|m+\alpha|^{\lambda}}\leq \sum_{\substack{Q(m,n)=k' \\ |m|>2|\alpha|}} \frac{|2\Delta|^{\lambda}}{|m|^{\lambda}}+ \sum_{\substack{Q(m,n)=k' \\ |m|\leq 2|\alpha|
 \\  m\neq\alpha }} \frac{|\Delta|^{\lambda}}{|m+\alpha|^{\lambda}}.
\end{aligned}
\end{equation}
The conditions $Q(m,n)=k', \ |m|>2|\alpha|$ imply that $(m,n)\in A_{Q,k'}$. For each fixed $m$ there can be at most two values of $n $ with $Q(m,n)=k'$, therefore the cardinality of pairs $(m,n)$ in the last sum is bounded by $8|\alpha|$. Therefore
\begin{equation}\label{3.1}
\begin{aligned}
\leq |2\Delta|^{\lambda}\sum_{(m,n)\in A_{Q,k'} } \frac{1}{|m|^{\lambda}}+ 8|\alpha||\Delta|^{\lambda}.
\end{aligned}
\end{equation}
We therefore need boundedness of the sum over $ A_{Q,k'}$. For $Q$  a positive definite form the equation (40) of  \cite{ts}   bounds this sum by  a constant $C_{\lambda,\Delta}$.  For negative definite $Q$ boundedness follows from the identity  $A_{Q,k'}=A_{-Q,-k'}$.   For a form of positive nonsquare discriminant it
is given by the equations (50),(51) of \cite{ts}. 
 Therefore, \eqref{3.1} is bounded by a constant  $C_{\lambda,\Delta,\alpha}.$

For  $1<p<\infty$   the Hölder inequality, and a decomposion via the sets $A_k$ gives
\begin{equation*}
\begin{aligned}
\|\mathcal{I}_{\lambda}f\|_{l^p(\mathbb{Z})}^p \leq C_{\lambda,p}\sum_{n\in \mathbb{Z}}\Big[\sum_{m\in \mathbb{Z}_*} \frac{|f(P(m,n))|^p}{|m|^{\lambda'p/2 }}\Big] \leq C_{\lambda,p}\sum_{k\in \mathbb{Z}} |f(k)|^p \sum_{(m,n)\in A_k} \frac{1}{|m|^{\lambda'p/2 }}.
\end{aligned}
\end{equation*}
We have seen that the inner sum is bounded by a constant depending on $\lambda,p,\Delta,\alpha.$ This concludes the case of negative or positive nonsquare discriminant.

When the discriminant is a positive square, we start with removal of some terms from the sum. 
\begin{equation*}
\begin{aligned}
\|\mathcal{I}_{\lambda}f\|_{p}^p&\leq \sum_{n\in \mathbb{Z}}\Big[\sum_{\substack{m\in \mathbb{Z}_*\\ P(m,n)=\Gamma}} \frac{|f(\Gamma)|}{|m|^{\lambda }} +\sum_{\substack{m\in \mathbb{Z}_*\\ P(m,n)\neq \Gamma}} \frac{|f(P(m,n))|}{|m|^{\lambda }}  \Big]^p \\ &\leq 2^{p-1}|f(\Gamma)|^p \sum_{n\in \mathbb{Z}}\Big[\sum_{\substack{m\in \mathbb{Z}_*\\ P(m,n)=\Gamma}} \frac{1}{|m|^{\lambda }} \Big]^p +2^{p-1} \sum_{n\in \mathbb{Z}}\Big[\sum_{\substack{m\in \mathbb{Z}_*\\ P(m,n)\neq \Gamma}} \frac{|f(P(m,n))|}{|m|^{\lambda }}\Big]^p  .
\end{aligned}
\end{equation*}
We first handle the second sum, which, as will be seen,  is bounded    whenever $\lambda>1-p^{-1}$. 
We apply the Hölder inequality, and then decompose
\begin{equation*}
\begin{aligned}
\leq C_{\lambda,p}\sum_{n\in \mathbb{Z}}\Big[\sum_{\substack{m\in \mathbb{Z}_*\\P(m,n)\neq \Gamma}} \frac{|f(P(m,n))|^p}{|m|^{\lambda'p/2 }}\Big]\leq C_{\lambda,p}\sum_{k\in \mathbb{Z}\setminus\{\Gamma\}} |f(k)|^p \sum_{(m,n)\in A_k} \frac{1}{|m|^{\lambda'p/2 }}.  
\end{aligned}
\end{equation*}
It remains to bound the  inner sum. It satisfies
\begin{equation*}
\begin{aligned}
\leq  \sum_{\substack{Q(m,n)=k'\\m\neq -\alpha}} \frac{|\Delta|^{\lambda'p/2}}{|m+\alpha|^{\lambda'p/2}}&\leq \sum_{\substack{Q(m,n)=k' \\ |m|>2|\alpha|}} \frac{|2\Delta|^{\lambda'p/2}}{|m|^{\lambda'p/2}}+ \sum_{\substack{Q(m,n)=k' \\ |m|\leq 2|\alpha|\\ m\neq -\alpha  }} \frac{|\Delta|^{\lambda'p/2}}{|m+\alpha|^{\lambda'p/2}}\\ &\leq 
|2\Delta|^{\lambda'p/2}\sum_{(m,n)\in A_{Q,k'} } \frac{1}{|m|^{\lambda'p/2}}+ 8|\alpha||\Delta|^{\lambda'p/2}.
\end{aligned}
\end{equation*}
The condition $k\neq \Gamma$  ensures $k'\neq 0$,  we can therefore use \eqref{2.5} to conclude that this is bounded by a constant depending on $\lambda,p,\Delta,\alpha.$ 

To evaluate the first sum we recall  from our exploration above
\[\{(m,n)\in \mathbb{Z}^2: \ P(m,n)=\Gamma\}=\{(m,n)\in \mathbb{Z}^2: \ Q(m-\alpha/\Delta,n-\beta/\Delta)=0\}.\]
If we define
\[\gamma_1=\frac{a_2\alpha+c_1\beta}{\Delta},    \ \ \ \ \     \gamma_2=\frac{a_1\alpha+c_2\beta}{\Delta},  \]
and  recall the factorization of $Q$ achieved in section 2 
\[  =\{(m,n)\in \mathbb{Z}^2:  a_2m+c_1n=\gamma_1 \} \cup  \{(m,n)\in \mathbb{Z}^2:  a_1m+c_2n=\gamma_2 \}.   \]
 Here it becomes clear that $P(m,n)=\Gamma$ is solvable in integers if and only if at least one of $\gamma_1/g_1,\gamma_2/g_2  $ is an integer, in which case it has infinitely many solutions equally spaced on a line.  Once we have a solution it is immediate that $\Gamma$ and therefore  $Q(e,-d)/\Delta$ is an integer. We remark that the converse is not true, that is   $Q(e,-d)/\Delta \in \mathbb{Z}$ does not imply that $P(m,n)=\Gamma$ is solvable.  This can be seen from the example $P(m,n)=4m^2-4n^2-4n$ for which $\Gamma=Q(e,-d)/\Delta=1$, but clearly $P(m,n)=1$ is not solvable.
 
 Hence if $\gamma_i/g_i, \ i=1,2$ are both nonintegers  the first sum contributes zero, and $\lambda>1-p^{-1}$ is a sufficient condition.
 If  at least one of $\gamma_i/g_i, \ i=1,2$ is an integer,  then we further  assume $\lambda>p^{-1}$,  and  treat the first sum as follows
\[\leq 2^{p-1}\sum_{n\in \mathbb{Z}}\sum_{\substack{m\in \mathbb{Z}_*\\ P(m,n)=\Gamma}} \frac{1}{|m|^{\lambda p }}  \leq 2^{p-1}\sum_{m\in \mathbb{Z}_*}\sum_{\substack{n\in \mathbb{Z}\\ P(m,n)=\Gamma}} \frac{1}{|m|^{\lambda p }} \leq 2^{p}\sum_{m\in \mathbb{Z}_*} \frac{1}{|m|^{\lambda p }}\leq C_{\lambda,p}.  \]

It remains to prove the case $p=1$ and $P(m,n)=\Gamma$ not solvable, but this follows immediately from the arguments already expounded.

\end{proof}


\section{Uniform estimates}

In this section we will prove Theorem 3, and as an application of it obtain Theorem 4.
Our main tool in this will be the large sieve, which we  describe concisely.  It arises from orthogonality estimates on additive characters, and was first proposed by Linnik \cite{lin}, to be greatly developed  by subsequent work, see \cite{ik} for details.  Let $\mathcal{M}$ be a finite set of integers contained in an interval of length $M\geq 1$, and $\mathcal{P}$ be a  set of primes. For each $p\in \mathcal{P}$ let $\Omega_p \subset \mathbb{Z}/p\mathbb{Z}$ be a set of residue classes with  cardinality less than $p$.  We define 
\[   H:=\sum_{p\in \mathcal{P}\cap[1,Q]}\frac{\#\Omega_p}{p-\#\Omega_p} .  \]
The large sieve gives  the estimate
\[\#\{ m\in \mathcal{M}:  \  m \ (\text{mod} \ p)\notin \Omega_p   \ \ \text{for all}  \ \   p \in  \mathcal{P}              \}   \leq \frac{M+Q^2}{H}.    \]
It is very reasonable and common to choose $Q=\sqrt{M}$, and we will do so as well.

In order to implement the large sieve we will need the prime number theorem in arithmetic progressions. The number of primes $p\leq x$ is denoted by $\pi(x)$, and    the number of primes $p\leq x$ with $p\equiv a \ (\text{mod} \ q)$ is denoted by $\pi(x;q,a).$  We note that as $\gcd(q,a)$ must divide $p$,  unless $\gcd(q,a)=1$ only primes    that may satisfy $p\equiv a \ (\text{mod} \ q)$    are the prime factors of $q$.
So all other primes reside in $\phi(q)$ residue classes  given by $a$  prime to $q$.   Here $\phi(q)$ is the Euler totient function.
   The prime number theorem is
\[\lim_{x\rightarrow \infty}  \frac{\pi(x)}{x/\log x} =1.\]
The prime number theorem in arithmetic progressions elaborates on this result by  showing that  these primes are distributed equally among the $\phi(q)$ equivalence classes.
\[\lim_{x\rightarrow \infty}   \frac{\pi(x;q,a)}{x/\log x}=\frac{1}{\phi(q)},\]
 This theorem guarantees  the existence of a constant $C_{q,a}$ such that  for $x$  not less than this constant 
$  \pi(x;q,a)\geq x/2\phi(q)\log x.$
Defining $C_q=\max_a C_{q,a}$, for $x$ not less than this constant we have  this inequality uniformly in $a.$

We will make extensive use of the theory of quadratic residues as we apply the large sieve. We  therefore  briefly state the essentials of  this theory. For a complete treatment see \cite{ld}. An integer $m$  prime to an integer $n\geq 2$ is a quadratic residue of $n$ if  $x^2\equiv m \  (\text{mod} \ n)$ is soluble, otherwise it is   a quadratic nonresidue of $n$.  Henceforth we  concentrate mostly on quadratic residues of odd primes $p$, and use the terms residue and nonresidue to mean quadratic residue and quadratic nonresidue.  As  $x^2\equiv (p-x)^2  \ (\text{mod} \ p)$, there are at most $(p-1)/2$  residues, but as a degree $r$ congruence  in prime modulus has at most $r$ solutions, there must be exactly $(p-1)/2$ residues, and thus $(p-1)/2$  nonresidues as well. We define  Legendre's symbol $(m|p)$ for $m$ not divisible by $p$ as equal to $1$ if $m$ is a residue, and to $-1$  if it is a nonresidue. We observe that as    $x^2\equiv m \  (\text{mod} \ p)$ and $y^2\equiv n \ (\text{mod} \ p)$ imply $(xy)^2\equiv  mn \ (\text{mod} \ p)$, the product of two residues is a residue, and as $x\mapsto mx$ is an automorphism  of the group  $\mathbb{Z}/p\mathbb{Z}$, the product of a residue with a nonresidue must be a nonresidue. Finally, this last argument, implemented with automorphisms induced by  nonresidues implies that the product of two nonresidues is a residue. Thus  Legendre's symbol satisfies  $(mn|p)=(m|p)(n|p).$

Jacobi's symbol extends Legendre's symbol to nonprime moduli. For $n$ any positive odd number, and $m$ an integer  prime to $n$ we define $(m|1)=1$ for $n=1$, and $(m|n)=(m|p_1)(m|p_2)\ldots(m|p_l)$ for $n=p_1p_2\ldots p_l$ with $p_i, \ 1\leq i\leq l$ being odd primes not necessarily distinct.  When $m$ is a  residue of $n$ it is a  residue of each prime factor of $p$, therefore   $(m|n)=1$. But the converse is not true, 
as when an even number of $(m|p_i)$ are negative we immediately have $(m|n)=1$. 
The fundamental result of the theory of quadratic residues is the quadratic reciprocity law, and with Jacobi's symbol at hand we can state a general version of it.  For  positive, odd, relatively prime integers $m,n$  we have 
\[(m|n)(n|m)=(-1)^{\frac{(m-1)(n-1)}{4}}.\]
We also note the particular cases  $(2|n)=(-1)^{(n^2-1)/8}$ and  $(-1|n)=(-1)^{(n-1)/2}$ of Jacobi's symbol that  will be of use below.

We will face, in the course of  proof of Theorem 3,   two fundamental issues regarding quadratic residues. The first of these is to know which odd primes $p$ make a fixed integer $m$ a quadratic nonresidue.   This concerns the multiplicative structure of residues, which, as we already have caught a glimpse of, is rather rich.  We will exploit this via  Chinese remainder theorem that we recall now. Let $m_1,m_2,\ldots,m_l$ be positive integers any two of which are relatively prime, and let $m$ be their product.  Then $x+m\mathbb{Z}\mapsto (x+m_1\mathbb{Z},x+m_2\mathbb{Z},\ldots, x+m_l\mathbb{Z})$ is a ring isomorphism from
$\mathbb{Z}/m\mathbb{Z}$  to $\mathbb{Z}/m_1\mathbb{Z}\times \mathbb{Z}/m_2\mathbb{Z}\times \cdots\times  \mathbb{Z}/m_l \mathbb{Z}$.  Therefore this map is also a group isomorphism between multiplicative groups 
$(\mathbb{Z}/m\mathbb{Z}) ^*$ and $(\mathbb{Z}/m_1\mathbb{Z})^*\times( \mathbb{Z}/m_2\mathbb{Z})^*\times \cdots\times  (\mathbb{Z}/m_l \mathbb{Z})^*.$  

A perfect square is always a residue.  A number cannot be a residue or nonresidue for its prime factors.  Hence,  for a fixed nonsquare integer $m$ we are looking for odd primes $p$ that yield $(m|p)=-1$.  Letting $m=(-1)^i|m|$ with $i$ being zero or one,  $(m|p)=(-1|p)^i(|m||p)$.  Writing $|m|=st^2$ with $s$  squarefree we obtain $(|m||p)=(s|p),$  and letting  $s=2^jr$ with $j$ zero or one, and $r$ odd,   $(s|p)=(2|p)^j(r|p)$. Now applying quadratic reciprocity 
\[(m|p)=(-1)^{\frac{(r-1)(p-1)}{4}}(-1|p)^i(2|p)^j (p|r) =(-1)^{i\frac{p-1}{2}+j\frac{p^2-1}{8}+\frac{(r-1)(p-1)}{4} } (p|r)    \]    

We first investigate the exponent of $-1$.  We only need to know  its value in modulus 2.  As $p$ is an odd prime, it is congruent to one of   $p_0=1,3,5,7$ in modulus 8.   As $r$ is odd, it is congruent to one of $r_0=1,3$ in modulus 4.  Thus in modulus 2 the exponent of $-1$ is congruent to 
\[ i\frac{p_0-1}{2}+j\frac{p_0^2-1}{8}+\frac{r_0-1}{2}\frac{p_0-1}{2}.      \]  
We calculate this for all 8 possibilities of $(i,j,r_0)$ in modulus 2 to obtain 

{\bf 1.} for $(0,0,1),(1,0,3),$ zero for all four values of $p_0,$
 
 {\bf 2.} for the remaining six triples,  zero for two  values of $p_0$, and one for the other two values.
  
  It remains to compute $(p|r)$. When $r=1$ this is by definition 1. We note that when this happens $(i,j,r_0)$ cannot be the triples  considered in  {\bf 1},  for $(0,0,1)$ together with $r=1$ would mean that $m$ is a perfect square, and $(1,0,3)$ is inconsistent with $r=1.$ Hence for $r=1$ we can conclude that $(m|p)=1$ for $p$ in  two equivalence classes of $8$, and $(m|p)=-1$ for $p$ in  the other two  equivalence classes. 
  
  We now assume $r>1$, in which case it can be factorized    into distinct odd primes $r=r_1r_2\cdots r_l$.   As $\gcd(p,r)=1$, there are $\phi(r)$ equivalence classes  of modulus $r$, and  $\phi(r_i)=r_i-1$  classes  of modulus $r_i$,  to which $p$ may belong. Since Euler's totient function is multiplicative 
  $\phi(r)=\phi(r_1)\phi(r_2)\cdots\phi(r_l).$ 
  The definition of Jacobi's symbol gives 
  $(p|r)=(p|r_1)(p|r_2)\cdots(p|r_l)$.  If we fix the value of one factor $(p|r_i)$  as plus or minus one there are $(r_i-1)/2$ equivalence classes of modulus $r_i$ to which 
  $p$ may belong. So if the value of every factor is fixed, by Chinese remainder theorem,  there are $2^{-l}\phi(r)$ equivalence class of modulus $r$ to which $p$ may belong.  For $(p|r)$ to be one, an even number of factors  should be fixed as $-1$, this gives rise to 
  \[{l\choose{0}} +{ l\choose{2}} +{l\choose{4}}+ \cdots+{ l\choose{2\lfloor   l/2\rfloor }}  =2^{l-1} \]  
  choices, and thus to $2^{-1}\phi(r)$ equivalence classes of $r$. For $(p|r)$ to be $-1$ we again have $2^{-1}\phi(r)$  classes.

  Thus we  conclude via Chinese remainder theorem  that  in  modulus $8r$ for $p$ in  $2\phi(r)$ equivalence classes   $(m|p)=1$, and for $p$ in $2\phi(r)$ equivalence classes $(m|p)=-1$. 
This conclusion combined with the prime number theorem in arithmetic progressions means  that essentially half of all primes make  a nonsquare $m$ a quadratic residue, and half make it a quadratic nonresidue. 

The second issue  regarding quadratic residues we need to understand   concerns their additive structure.  Specifically,   if      the set of quadratic nonresidues of  a prime is shifted by a fixed integer, how many of them are still nonresidues?  This problem was answered in full by Perron in his article \cite{op}. Results depend on the value of the prime in modulus 4.  Let $p$ be a prime  and let $N_p$ denote its set of nonresidues. Let $s$ be an integer relatively prime to $p.$ Then the set $N_p+s\cap N_p$ has cardinality

{\bf 1.}  $k-1$ if $p=4k-1$, 

{\bf 2a.}  $k$ if $p=4k+1$ and $s\notin N_p $, 

{\bf 2b.}  $k-1$ if $p=4k+1$ and $s\in N_p $, \\
Thus we  conclude  that  for any $p$ the set in question  has cardinality at least $(p-5)/4$. 

We are now ready to prove Theorem 3. We first estimate the  density of  $m$ that  satisfy  $am^2+n^2=N$ within an interval, and then  using dyadic decomposition apply this to bound the sum  \eqref{1.5}. 
To obtain the density result  we apply the large sieve by showing that for primes $p$ comprising a  sufficiently large set   chosen via the theory of quadratic residues,  there are roughly $p/2$ equivalence classes  elements of which cannot be in the set of $m$ mentioned above.  

\begin{proof}
Let  $\mathcal{M}_{\tau}:=(\tau-M,\tau+M)\cap \mathbb{Z}$ for an arbitrary integer  $\tau$. We want to estimate the cardinality of  the set 
\[ {E}_{\tau}:= \{m\in \mathcal{M}_{\tau} :   am^2+n^2=N \ \   \text{for some }   n\in \mathbb{Z} \} . \]
As we have $  n^2=N-am^2$,  for $m$ to be in  $E_{\tau}$ the term $N-am^2$ must not be a quadratic nonresidue in any modulus.  Let $\mathcal{P}$ be the set of primes  greater than 5 for which $(-a|p)=-1.$ 
 For  these primes  we will establish the existence of  
 equivalence classes $\Omega_p$  elements $m$ of which make $N-am^2$ nonresidue in modulus $p$. We have two cases depending on whether $p$ divides $N.$ 

First assume $p$ does divide $N$.  Then   $ N-am^2\equiv -am^2 \ (\text{mod} \ p)$.  
 As $-a$ is a nonresidue and $m^2$ is a residue unless $m\equiv 0 \ (\text{mod} \ p)$, their product is a nonresidue unless $m\equiv  0 \ (\text{mod} \ p)$. Therefore   $\Omega_p$ contains every equivalence class except 0.

Now assume $p$ does not divide $N.$ As the set   $\{-am^2  \ (\text{mod} \ p):  m\not \equiv 0  \ (\text{mod} \ p)\} $ gives the set of   nonresidues  of $p$, the  set $\{N-am^2  \ (\text{mod} \ p):  m\not \equiv 0  \ (\text{mod} \ p)\} $ represents shifting of these by a number  prime to $p$. Therefore it contains at least $(p-5)/4$ nonresidues, for each of which we have two incongruent values of $m$. So $\Omega_p$ contains at least $(p-5)/2$ equivalence classes.

The set $E_\tau$ is contained in  
$\{m\in \mathcal{M}_{\tau} :  m \ (\text{mod} \ p )\notin \Omega_p  \ \text{for  all } \ p\in \mathcal{P}   \} $, and the cardinality of this set can be estimated by applying the large sieve.  We 
consider the elements of $\mathcal{P}$ bounded by $\sqrt{2M}.$  
Letting $r$ be the odd part of squarefree part of  $|a|$, 
our exposition of  quadratic residues makes it clear that $\mathcal{P}$ is the set of primes   that reside in $2\phi(r)$ equivalence classes of $8r$,  exceed 5,  and does not divide $-a$.  So,  for $M\geq K_a:= [8(3+2\log|a|)]^4+C_{8r}^2$  we can estimate 
\[\# \mathcal{P}\cap [1,\sqrt{2M}] \geq   \frac{\sqrt{2M}}{4\log{\sqrt{2M}}} -(3+2\log|a|) \geq   \frac{\sqrt{2M}}{8\log{\sqrt{2M}}} \geq   \frac{\sqrt{M}}{4\log{M}}.\]
    We easily calculate  $H\geq 6^{-1}\# \mathcal{P}\cap [1,\sqrt{2M}]. $ Therefore
\[ \#\{m\in \mathcal{M}_{\tau} :  m \ (\text{mod} \ p )\notin \Omega_p  \ \text{for  all } \ p\in \mathcal{P}   \} \leq \frac{4M}{H}\leq 10^2\sqrt{M}\log M .  \]
As seen clearly this bound is independent of $\tau, N.$

With this result at hand we can  proceed to  \eqref{1.5}.  Consider a decomposition of $\mathbb{Z}-\{\tau\}$ into dyadic subsets 
$D_j:=  \{m\in \mathbb{Z}:    \  2^{j-1} \leq  |m-\tau|  < 2^{j}  \}  $ for $j\in \mathbb{N}$.
Let $j_0$ be such that $2^{j_0-1} \leq  K_a< 2^{j_0} $.
With  these we can write
\begin{equation*}
\begin{aligned}
\sum_{\substack{(m,n)\in\mathbb{Z}-\{\tau\}\times \mathbb{Z} \\  am^2+n^2=N   }  }\frac{1}{|m-\tau|^{\lambda}}& =\sum_{j=1}^{j_0}\sum_{\substack{(m,n)\in D_j\times \mathbb{Z} \\  am^2+n^2=N   }  }\frac{1}{|m-\tau|^{\lambda}}  +\sum_{j>j_0}^{\infty}\sum_{\substack{(m,n)\in D_j\times \mathbb{Z} \\  am^2+n^2=N   }  }\frac{1}{|m-\tau|^{\lambda}}    \\  &\leq 8K_a+200\sum_{j=1}^{\infty} \frac{2^{j/2} \log 2^j }{2^{\lambda(j-1)}}, 
\end{aligned}
\end{equation*}
and this is bounded by a constant depending only on $a,\lambda$.

\end{proof}

 Theorem 4 concerns polynomials of negative or positive nonsquare discriminant, and by completing squares representation by these can be connected to representation by diagonal forms.  Indeed, if $q(x,y)=ax^2+bxy+cy^2$ represents $k$ with $(m,n)$,  then multiplying both sides by $4c$
 \[am^2+bmn+cn^2=k \ \ \ \implies \ \ \     -\Delta m^2+(bm+2cn)^2=4ck.                                                            \]
So the form $q'(x,y)=-\Delta(q) x^2+y^2$ represents $4ck$ with $(m,bm+2cn)$. Further, when $c\neq 0$ the map $(m,n)\mapsto (m,bm+2cn)$  is injective. The proof of Theorem 4 is similar to the proof of Theorem 2, but employs this connection  to apply Theorem 3, instead of decomposition arguments and bounds from the work \cite{ts}.

\begin{proof}
We first consider $p=1$ case,  which we reduced to bounding \eqref{1.2}. As usual $Q$ denotes the form corresponding to the polynomial $P$. Let 
$Q'(x,y)=-\Delta(P)x^2+y^2$. Combining  \eqref{y3.1} with the relation  we just described about   representation by $Q$ and $Q'$ we obtain 
\[ \sum_{(m,n)\in A_k} \frac{1}{|m|^{\lambda }}\leq \sum_{\substack{Q(m,n)=k'  \\ m\neq -\alpha}} \frac{|\Delta|^{\lambda}}{|m+\alpha|^{\lambda}}  \leq \sum_{\substack{Q'(m,n)=4ck'  \\ m\neq -\alpha}} \frac{|\Delta|^{\lambda}}{|m+\alpha|^{\lambda}}. \]
By Theorem 3 this last sum is bounded by a constant that depends only on $\lambda,\Delta$ when $\lambda>1/2.$

When $p>1$,  let $\lambda''=\lambda-(1-(2p)^{-1})$. By the Hölder inequality we have
\begin{equation*}
\begin{aligned}
\|\mathcal{I}_{\lambda}f\|_{l^p(\mathbb{Z})}^p\leq     \sum_{n\in \mathbb{Z}}\Big[\sum_{m\in \mathbb{Z}_*} \frac{|f(P(m,n))|^p}{|m|^{[(p^{-1}+\lambda'')/2]p }}\Big]\Big[\sum_{m\in \mathbb{Z}_*}\frac{1}{|m|^{[\lambda-(p^{-1}+\lambda'')/2]p' }}\Big]^{p-1}.
\end{aligned}
\end{equation*}
The second sum is clearly finite, and depend only on $\lambda,p$. As for the first sum, we again perform a decomposion using the sets $A_k$,
\begin{equation*}
\begin{aligned}
\leq C_{\lambda,p}\sum_{k\in \mathbb{Z}} |f(k)|^p \sum_{(m,n)\in A_k} \frac{1}{|m|^{(1+\lambda''p)/2 }}.
\end{aligned}
\end{equation*} 
As  $(1+\lambda''p)/2$  exceeds $1/2$,  the inner sum depends only on $\lambda,p,\Delta.$ This concludes the proof.

\end{proof}


\section{Lattice points and diophantine approximation}

In this section we explore lattice points on conics via their connections to diophantine approximation.  We  first prove Theorem 5. In its proof our main tool is Schmidt's theorem on simultaneous diophantine approximation  \cite{sch} which we now recall.  Let $\theta_1, \theta_2,\ldots, \theta_l$ be real algebraic numbers such that $1,\theta_1, \theta_2,\ldots, \theta_l$ are linearly
independent over the rationals. Then  for every $\varepsilon> 0$ there are only finitely many positive integers $q$ with
\[q^{1+\varepsilon }\|q\theta_1\|  \|q\theta_2\|\cdots \|q\theta_l\|   <1, \] 
where $\|\cdot\|$ is the distance to the nearest integer.
 In order to fulfill the hypothesis of  the theorem we will appeal to the well known fact, due to  Besicovitch \cite{bes}, that the set of square roots of squarefree natural numbers is linearly independent over the rationals.

To prove Theorem 5 we  obtain lattice points on hyperbolas with asymptotes of irrational algebraic  slopes. This part of the proof is very similar to Chan's works \cite{thc1,thc2}, although he views these hyperbolas as simultaneous Pell equations.   Then lattice points yield very close rational approximation of slopes of asymptotes, and this contradicts Schmidt's theorem.

\begin{proof}
Suppose   $N\in S'$ with  $N=R^2+r$ and $  R^{\rho\varepsilon}\geq h_l^{100l}$ where we define $\varepsilon:=4^{-1}\min\{1,2\rho(l-1)-1\}$.  Then $(R-h_i)^2+n_i^2=N$ is equivalent to $2Rh_i=h_i^2+n_i^2-r$, and this equation implies $Rh_i< n_i^2<3Rh_i $, and again from it for $i< j$ we obtain the equations
\[
h_i(h_j^2+n_j^2-r)= 2Rh_ih_j=h_j(h_i^2+n_i^2-r).
 \]
 We rearrange these as 
 \begin{equation}\label{5.1}
 h_in_j^2-h_jn_i^2=(h_i-h_j)(h_ih_j+r), 
 \end{equation}
 and let $h_{ij}$ denote the right hand side in this equation. We have $|h_{ij}|\leq 2h_j^3R^{1/2-\rho}< \sqrt{R}.$
 
 The equations \eqref{5.1} provide lattice points on the hyperbolas $h_ix^2-h_jy^2=h_{ij}$, and the slopes of asymptotes of these hyperbolas are given by $\pm\sqrt{h_i/h_j}$.  The hyperbolas are given by the graphs of the functions $y=\pm    \sqrt{(h_ix^2-h_{ij})/h_j  }$,
and with these we estimate the distance between the hyperbolas and their asymptotes at the points supplied by \eqref{5.1}.
  
\begin{equation*}
\begin{aligned}
\|\sqrt{h_i/h_j}n_j\| \leq |n_i-\sqrt{h_i/h_j}n_j|  
 & =  |  \sqrt{(h_in_j^2-h_{ij})/h_j  }  -    \sqrt{h_i/h_j}n_j  |  \\ 
&=  |h_{ij}/h_j|\cdot  [     \sqrt{(h_in_j^2-h_{ij})/h_j  }              +    \sqrt{h_i/h_j}n_j  ]^{-1}     \\  &\leq    |h_{ij}/h_j|\cdot  [ \sqrt{h_i/h_j}n_j  ]^{-1},
\end{aligned}
\end{equation*}
which is bounded by $ 2h_l^2R^{-\rho}.$  When $|r|\leq h_l^2$ this bound improves to $  2h_l^2R^{-1/2}.$

We have three cases depending on $r$. The first is when    $r\neq -h_ih_j$ for any $i\neq j$.  We  show that the linear independence hypothesis in Schmidt's theorem is fulfilled.    Let $c_i$ be rational coefficients and consider a linear combination  
\[  c_1\sqrt{h_1/h_l} +  c_2\sqrt{h_2/h_l}+\cdots +c_{l-1}\sqrt{h_{l-1}/h_l}+c_l=0.\]
 Multiplying both sides by     $\sqrt{h_l}$, and then writing $h_i=s_it_i^2$ with $s_i$ squarefree  we obtain 
\[  c_1t_1\sqrt{s_1} +  c_2t_2\sqrt{s_2}+\cdots +c_{l-1}t_{l-1}\sqrt{s_{l-1}}+c_lt_l\sqrt{s_{l}}=0.\]
By Besicovitch's  result this implies that all $c_i$ are zero unless $s_i=s_j$ for some $i<j$.  But this is not possible, as it would imply from  \eqref{5.1}
\[s_i|(t_in_j)^2-(t_jn_i)^2|= |(h_i-h_j)(h_ih_j+r)|< \sqrt {R}, \]
where  $r\neq -h_ih_j$ means   $t_in_j\neq t_jn_i$, and thus   the left hand side is at least $2\sqrt{R}$. Hence we established the independence of $1,\sqrt{h_1/h_l}, \ldots, \sqrt{h_{l-1}/ h_l}$ over the rationals.
As we have
\begin{equation*}
 n_l^{1+\varepsilon}\prod_{i=1}^{l-1}\|\sqrt{
h_i/h_l}n_l\| \leq  (3h_lR)^{\frac{1+\varepsilon}{2}}(2h_l^2R^{-\rho})^{l-1}\leq h_l^{4l}R^{-\rho(l-1)+\frac{1+\varepsilon}{2}}  \leq R^{\varepsilon (\frac{\rho}{25}-\frac{3}{2})},
\end{equation*}
when Schmidt's theorem is applied to $\sqrt{h_i/h_l}, \ 1\leq i\leq l-1$ with the value of $\varepsilon$ fixed at the beginning, this $n_l$ is one of the finite number of exceptions.

The second case is $r=-h_ih_j$ for some $i<j$ and $r<-h_{\lceil l/2 \rceil}^2$. We observe that $r\neq -h_ih_j$ if  $1\leq i<j\leq 3$. This, as in the first case, ensures that $1,\sqrt{h_1/h_3},\sqrt{h_2/h_3}$ are linearly independent over the rationals.  Further 
\begin{equation*}
 n_3^{1+\varepsilon}\prod_{i=1}^{2}\|\sqrt{h_i/h_3}n_3\| \leq  (3h_lR)^{\frac{1+\varepsilon}{2}}(2h_l^2R^{-1/2})^{2}\leq h_l^{9}R^{-3/8} \leq R^{-{1}/{4}}.
\end{equation*}
Thus  when Schmidt's theorem is applied to $\sqrt{h_i/h_3}, \  i=1,2$ with the value of $\varepsilon$ fixed at the beginning, this $n_3$ is one of the finite number of exceptions. 

The third case is $r=-h_ih_j$ for some $i<j$ and $r\geq -h_{\lceil l/2 \rceil}^2$.  When we apply the arguments of the second case to $h_{l-2},h_{l-1},h_l$, we observe $n_l$ to be an exception. 

 Infiniteness of   $S'$  would yield  infinitely many $N\in S'$ with $R^{\rho\varepsilon}\geq h_l^{100l}$  in one of the three cases above. From these $N$  using $Rh_i< n_i^2<3Rh_i $ we can extract infinitely many exceptions,  violating    Schmidt's theorem. Hence $S'$ must be  finite.

\end{proof}

We remark that, as the proof reveals, if we remove $h_ih_j, \ i<j$ from the set at the outset we can drop the condition $l\geq5$. In particular if we take $S$ to be the set of squares  $l=3$ is sufficient. We also remark that methods of this proof should yield a similar result for the hyperbola $x^2-y^2$.   Lastly, observe that ineffectiveness of Schmidt's theorem forces us to fix $h_i, \ 1\leq i \leq  l$ beforehand, and this precludes any progress towards the Conjecture 1.  Fielding effective results in diophantine approximation we next make some progress in this direction.

 We state the effective result we will use in the proof of Theorem 6 below. This result is recently obtained by Bugeaud \cite{yb}, and it improves upon the work of Turk \cite{jt}, employed by Chan in his work.  Let $a, b$ be positive integers such that none of $a, b,  ab$ is a full square, and  let $u, v$ be nonzero integers. Then, there exists an effectively computable, absolute
 real number $C$ such that all  solutions in positive integers $x, y, z$ of the  equations
\begin{equation*}
x^2-ay^2=u, \quad     \quad   \quad    \quad  z^2-by^2=v 
\end{equation*}
satisfy 
\begin{equation*}
\max\{ x,y,z\}\leq\big (\max\{|u|,|v|,2\}\big)^{C\sqrt{ab}\log a\log b}.
\end{equation*}

Having stated this result we are ready to prove Theorem 6. As in the proof of Theorem 5 we obtain simultaneous Pell equations, and we apply Bugeaud's theorem after verifying that these equations fulfill its hypotheses. Our proof differs from Chan's  in that we carry out this verification  in a simpler and more efficient way.

\begin{proof}
We start with  \eqref{1.6}. Let $N$ be large as described in the theorem.  Suppose for this $N$ the set in question contains more than 20 elements, this means there are five positive integers $\sqrt{|r|}<n_1<n_2<\ldots<n_5$ with $(m_i,n_i),  \ m_i>0$  in the set.  We let $m_i=R-h_i$, and note that $0<h_1<h_2<\ldots <h_5.$   From  $(R-h_i)^2+n_i^2=R^2+r$ we obtain the relations $h_i\leq 72\log^{\kappa/2} N$ and  $Rh_i< n_i^2<3Rh_i $,      and  for $i<j$ the equation
 \begin{equation}\label{5.2}
 h_in_j^2-h_jn_i^2=(h_i-h_j)(h_ih_j+r)=h_{ij},
 \end{equation}
with  $|h_{ij}|\leq 2^{5}h_j^3e^{2\log^{\kappa}N  }< \sqrt{R}.$ We  let $h_i=s_it_i^2$, with $s_i$ squarefree.

We have two cases, the first being $r<-h^2_3$. 
 In this case we consider the equations
 \begin{equation*}\label{}
 h_1n_3^2-h_3n_1^2=h_{13},  \quad     \quad   h_2n_3^2-h_3n_2^2=h_{23}. 
 \end{equation*}
 Multiplying both with $-h_3$ we obtain
  \begin{equation*}\label{}
(h_3n_1)^2-  h_1h_3n_3^2=-h_3h_{13},  \quad     \quad   (h_3n_2)^2-h_2h_3n_3^2=-h_3h_{23}, 
  \end{equation*}
 and observe that with
  \[(x,y,z)=(h_3n_1,n_3,h_3n_2),  \  \  \text{and}   \ \   \ (a,b,u,v)=(h_1h_3,h_2h_3,-h_3h_{13} ,-h_3h_{23} )\]
   this is a system of equations to which Bugeaud's theorem can be applied  if we can verify the conditions on $a,b,u,v$. Clearly our condition $r<-h^2_3$ implies $h_{ij}, \ 1\leq i<j\leq 3$ are nonzero, and thus $u,v$ are nonzero. Also observe that if one of $a,b,ab$  is a full square, then $s_i=s_j$ for some  $1\leq i<j\leq 3.  $  Then the equation \eqref{5.2} implies
 \[ s_i|(t_in_j)^2-(t_jn_i)^2|= |h_{ij}|< \sqrt {R} . \]
 As $h_{ij}$ are nonzero, $t_in_j\neq t_jn_i$, and this means the leftmost term  is at least $2\sqrt{R},$ leading to a contradiction. We therefore fulfilled the hypothesis of Bugeaud's theorem, on applying which
 \begin{equation*}
 \begin{aligned}
  \max\{  h_3n_1,n_3, h_3n_2 \}& \leq   \big(\max\{|h_3h_{13}|,|h_3h_{23}|,2\}\big)^{Ch_3\sqrt{h_1h_2}\log h_1h_3 \log h_2h_3}  \\ &\leq   \big(  e^{4\log^{\kappa}N}         \big)^{(72)^2C\log^{\kappa}N\log^2\log N}  \\ &\leq    e^{2^{16}C\log^{2\kappa}N\log^2\log N}.
  \end{aligned}
  \end{equation*}
Since the leftmost term  is  greater than $(N/2)^{1/4}$, this is a contradiction.

The second case is $r\geq -h_3^2$, and in this case the same argument with $h_3,h_4,h_5$ instead of $h_1,h_2,h_3$ yields the contradiction we  are looking for. Therefore the assumption that the set  contains more than 20 elements is wrong. 

The hyperbolic case \eqref{1.7} follows if we repeat the same arguments almost verbatim. 

As for  \eqref{1.8} we consider the injective map  $(u,v)=T(m,n)=(m+n,-m+n)$ from $\mathbb{Z}^2$ to itself.  The condition $|n-\sqrt{N}|\leq 2N^{1/4}\log^{\kappa/4}N  $ implies $|m-\sqrt{N}|\leq 4N^{1/4}\log^{\kappa/4}N  $. Combining these gives $|v|\leq 6N^{1/4}\log^{\kappa/4}N$, which means $T$ maps
 the set in \eqref{1.8} into the set 
\[\{ (u,v):  u^2-v^2=4N, \ \ \      |v|\leq 5(4N)^{1/4}\log^{\kappa/4}4N                           \}.\]
As  $4N\in E $, the cardinality of this set is at most $20$ by  \eqref{1.7}. Thus there are at most $5$ positive values of $u$, and to each of these correspond two values of $v$, which yield at most 10 pairs in \eqref{1.8}.

\end{proof}

The hypothesis imposed by Chan  \cite{thc2}  in his theorem on divisors is of the following form: Let $N$ be a sufficiently large integer that can be factored as $(M-a)(M+b)$ for integers $0\leq a\leq b\leq e^{\log^{\kappa}N }.$ We would like to illustrate the  reduction of his result under this condition to a subcase of  our \eqref{1.8}.  
 For such  $N$ if we let $R=2M+b-a$ we obtain 
$4N=R^{2}-(a+b)^2$, and thus observe that  $4N$ is of the form $R^2+r$ with $|r|\leq 4e^{2\log^{\kappa}N}$.  If $n$ is a divisor of $N$  with $|n-N^{1/2}|\leq N^{1/4}\log ^{\kappa/4}N$, then $2n$ is a divisor of $4N$ with $|2n-(4N)^{1/2}|\leq 2N^{1/4}\log ^{\kappa/4}N$. Therefore \eqref{1.8} applies to $4N$ to give at most  10 divisors $n$ of $N$.  
 
 Overall  Theorem 6 improves upon Chan's work in three respects. It reduces the cardinality of  the set of lattice points from $36$ and $18$ for the circle and divisor cases respectively to $20$ and $10.$ It enlarges the set of  possible values $N$ can take from points of the form $R^2\pm b^2$ to $R^2+r$. Finally by sharper computation and using Bugeaud's work the value of $\kappa$ is improved from $\kappa=2/7$ to $\kappa<1/2$.

\end{document}